\documentclass[reqno, 12pt]{amsart}
\usepackage{amsmath, amsthm, amssymb, mathtools, graphicx, paralist, cite, color}
\usepackage{comment}
\usepackage{hyperref}

 \usepackage{amsthm}
\usepackage[utf8]{inputenc}
\usepackage{mathtools}
\usepackage{amssymb}
\usepackage{amsmath, nccmath} 
\usepackage{amsthm}
\theoremstyle{remark}
\newtheorem*{remark}{Remark}
\theoremstyle{remark}

\theoremstyle{remark}

\theoremstyle{remark}

\theoremstyle{remark}
\newtheorem{nnremark}{Remark}
\theoremstyle{plain}
\newtheorem{theorem}{Theorem}
\newtheorem{proposition}{Proposition}
\newtheorem{lemma}{Lemma}

\usepackage{indentfirst}
\usepackage{appendix}
\usepackage{listings}
\usepackage{hyphenat}
\usepackage{physics}
\usepackage{calrsfs}
\usepackage{enumerate}
\usepackage{tikz-cd}
\usepackage{mathrsfs}
\usepackage{bbm}

\def \di {{d^{(0)}_{t}(f)}}
\def \dii {{d^{(1)}_{t}(f)}}

\title{Large Sieve Inequalities for periods of Maa{\ss} forms}
\author{Dimitrios Lekkas}
\address{ CoMENS Faculty, Northeastern University London, Devon House, 58 St Katharine's Way, London E1W 1LP}
\email{dimitrios.lekkas@nulondon.ac.uk}
\author{Marios Voskou}
\address{Department of Mathematics, University College London, Gower Street, London WC1E 6BT}
\email{marios.voskou.20@ucl.ac.uk}
\date{\today}

\date{\today}

\begin{document}
\begin{abstract} 
For $\Gamma$ a Fuchsian Group of the first kind, we obtain large sieve inequalities with weights the hyperbolic periods of Maa{\ss} forms of even weight. This is inspired by work of Chamizo, who proved a large sieve inequality with weights values of Maa{\ss} forms of weight $0$. The motivation is applications in counting problems in $\Gamma_1 \backslash \Gamma \slash \Gamma_2$, where $\Gamma_1$, $\Gamma_2$ are hyperbolic subgroups of $\Gamma$.
\end{abstract}
\maketitle
\section{Introduction}
In Analytic Number Theory, a large sieve inequality is an inequality that provides a non-trivial estimate for an exponential sum, with respect to the mean square of the coefficients. Large sieve inequalities are important, for example, in the study of averages of exponential sums, and spectral sums in general. A representative example is the following:
\begin{theorem}[Bombieri, \cite{bombieri}]
    Let $(a_n)_n$ be a sequence of complex numbers and let $x_1,\dots, x_R \in \mathbb{R} \slash \mathbb{Z}$ be such that $\left\| x_\mu- x_\nu\right\|>\delta$ for every $\mu \neq \nu$ and some $\delta>0$, where $\left\|x \right\|$ denotes the distance of $x$ from the closest integer. Then,
    $$ \displaystyle \sum_{\nu=1}^{R} \left| \sum_{n \leq N} a_ne^{2\pi i n x_\nu}\right|^2 \ll \left(N+\delta^{-1}\right) \sum_{n \leq N} \left|a_n\right|^2.$$
\end{theorem}
Of particular interest in modern analytic number theory is the harmonic analysis in $\Gamma \backslash \mathbb{H}$, where $\Gamma$ is a Fuchsian group of the first kind and $\mathbb{H}=\left.\left\{ z=x+iy \in \mathbb{C} \right\vert y>0 \right\}$, the upper half-plane. An example of a large sieve inequality in this setting is the one introduced by Jutila in \cite{jutila}. This was used by Nordentoft--Petridis--Risager in \cite{nordentoft}, to study the mean square of shifted convolution sums for
Hecke eigenforms. In \cite{chamizo1} and \cite{chamizo2}, Chamizo considered another large sieve inequality for $\Gamma \backslash \mathbb{H}$, more closely related to our work. We now describe the setup for his work.

We denote by $\mathfrak{h}_m$ the space of $L^2(\mathbb{H})$\hyp{}functions that transform as
$$F(\gamma z)=j^{2m}_{\gamma}(z)F(z)$$ under $\Gamma$, where for $\gamma =\begin{pmatrix}
a & b \\
c & d 
\end{pmatrix}$ we have $$j_{\gamma}(z)=\frac{cz+d}{\left|cz+d\right|}.$$
Let \begin{equation} \label{dms} D_m=y^2 \left ( \frac{\partial ^2}{\partial x^2} + \frac{\partial^2}{\partial y^2} \right) - 2imy\frac{\partial}{\partial x} \nonumber \end{equation} be the Laplacian in $\mathfrak{h}_m$.
Here, $L^2(\mathbb{H})$ is the space of functions $f$ such that $\left \langle f, \, f \right \rangle$ finite, where, for $f,g$ such that $f\cdot \overline{g}$ well-defined in $\Gamma \backslash \mathbb{H}$, we define
$$\left \langle f, \, g \right \rangle := \int_{\Gamma \backslash \mathbb{H}}f(z) \cdot \overline{g(z)} \; \,\frac{dx \, dy}{y^2}.$$
Fix a real-valued orthonormal $\mathfrak{h}_0$\hyp{}eigenbasis $\left(u_{0,j}\right)_j$ with respect to $D_0$, with corresponding eigenvalues $\left(\lambda_j\right)_j$.
Let also $E_{\mathfrak{a}}\left(z, \, s\right)$ denote the Eisenstein series with respect to the cusp $\mathfrak{a}$ (see \cite[(3.11)]{iwaniec}). Chamizo proved the following theorem. 

\begin{theorem}[Chamizo, \cite{chamizo1}] \label{chamizo}
Let $z \in \Gamma \backslash \mathbb{H}$.  Let $(a_n)_n$ be a sequence of complex numbers, and, for every cusp $\mathfrak{a}$, let $a_{\mathfrak{a}}(t)$ be a continuous function. Furthermore, let $T,X>1$ and $x_1, \dots , x_R \in [X,2X]$. If $|x_\nu-x_\mu| > \delta >0$ for $\nu \neq \mu$, then
   
$$
    \sum_{\nu=1}^R \Big| \sum_{|t_j|\le T}a_jx_{\nu}^{it_j}u_{0,j}(z)+\frac{1}{4\pi}\sum_{\mathfrak{a}}\int_{-T}^{T}a_{\mathfrak{a}}(t)x_{\nu}^{it}E_{\mathfrak{a}}(z,1/2+it)\, dt\Big|^2 \ll \big(T^2+XT\delta^{-1}\big)||\mathbf{a}||_*^2 \; ,
$$   

where $$||\mathbf{a}||_*= \Big(\sum_{|t_j| \le T}|a_j|^2+\frac{1}{4\pi}\sum_{\mathfrak{a}}\int_{-T}^{T}|a_{\mathfrak{a}}(t)|^2 \, dt \Big)^{1/2} \;.$$
\end{theorem}
Sums involving the exponential $X^{it_j}$ appear in the error term of hyperbolic counting problems, when estimated using spectral methods. When improvements in upper bounds for the error term are very difficult, we often consider its mean square instead. One famous example is the error term of the prime geodesic theorem. Cherubini and Guerreiro in \cite{cherubini} proved an upper bound for the mean square error of the prime geodesic theorem. Balog et al in \cite{balog} improved their result.

Large sieve inequalities involving the exponential $X^{it_j}$ are used in the study of the mean square of such error terms, as they provide cancellations when averaging over different values of $X$.  In particular, Chamizo uses Theorem \ref{chamizo} to provide upper bounds for the mean square of the error term in the classical hyperbolic lattice-counting problem (see \cite[Prop 2.1, Cor 2.1.1]{chamizo2}). Theorem \ref{chamizo} has also been used by Chatzakos--Petridis to prove an analogous upper bound for the mean square error term in the hyperbolic lattice problem in $\Gamma_1 \backslash \Gamma \slash \Gamma_2,$ where $\Gamma_1$ and $\Gamma_2$ are a hyperbolic and an elliptic subgroup of $\Gamma$ respectively (see \cite[Thm 1.2]{chatzakos}). Both of these problems are cases of the nine double-coset hyperbolic counting problems studied by A. Good in \cite{good}. 
In this paper, we prove a family of new large sieve inequalities for periods of Maa{\ss} forms, that will be crucial in upcoming results for the mean square error term in the case when $\Gamma_1$ and $\Gamma_2$ are both hyperbolic.

We define the Maa{\ss} raising operators by
$$K_m=(z-\bar{z})\frac{\partial}{\partial z}+m.$$
It can be shown that $K_m$ maps $\mathfrak{h}_m$ to $\mathfrak{h}_{m+1}$ (see \cite[p.308]{roelcke}). Furthermore, for any $m$, the functions $\left(u_{m,j}\right)_j$ defined recursively by $$u_{m+1,j}=\frac{i}{\sqrt{\lambda_j+m^2+m}} \cdot K_mu_{m,j}$$ form an orthonormal $\mathfrak{h}_m$\hyp{}eigenbasis for $D_m$ with the same corresponding eigenvalues (see \cite[ p.146, eq.11]{fay}).
In a similar manner, we define
$$E_{\mathfrak{a},0}\left(z, \, s\right):=E_{\mathfrak{a}}\left(z, \, s\right), \quad E_{\mathfrak{a},m+1}\left(z, \, s\right):=\frac{i}{\sqrt{1/4+t^2+m^2+m}} \cdot K_m E_{\mathfrak{a},m}\left(z, \, s\right).$$
Let $\Gamma_1$ be a hyperbolic subgroup of $\Gamma$, fixing the geodesic $l$ in $\Gamma \backslash \mathbb{H}$.
We further define the \emph{periods}
$$\hat{u}_{m,j}:=\int_{l}u_{m,j}(z) ds(z), \; \hat{E}_{\mathfrak{a}, m}(s'):=\int_{l}E_{\mathfrak{a},m}\left(z, \, s'\right) ds(z),$$
where $ds(z)$ is the Poincar\'e metric defined by $$ds(z)^2=y^{-2}\left( dx^2+dy^2 \right).$$
We prove large sieve inequalities analogous to Theorem \ref{chamizo},  with weights $\hat{u}_{m,j}$ instead of $u_{0,j}(z)$.
\begin{theorem} \label{sievevar}
    Let $m$ be a non-negative integer. Let $T,X>1$ and $x_1, \dots , x_R \in [X,2X]$. If $|x_\nu-x_\mu| > \delta >0$ for $\nu \neq \mu$, then
       \begin{equation*} \sum_{\nu=1}^R \Big| \sum_{|t_j|\le T}a_jx_\nu^{it_j}\widehat{u}_{m,j}+\frac{1}{4\pi}\sum_{\mathfrak{a}}\int_{-T}^{T}a_{\mathfrak{a}}(t)x_{\nu}^{it}\hat{E}_{\mathfrak{a}, m}(1/2+it) \, dt\Big|^2 \ll \big(T+X\delta^{-1}\big)||\mathbf{a}||_*^2,\end{equation*}
   where the implicit constant depends on the geodesic segment $l$ and the group $\Gamma$.
\end{theorem}
\begin{remark}
    As we prove in section \ref{pfofsieve}, it is easy to show the relations
\begin{equation} \label{highperiods} \hat{u}_{m+2,j}=-\sqrt{\frac{m^2+m+\lambda_j}{m^2+3m+2+\lambda_j}}\hat{u}_{m,j}, \quad \hat{E}_{\mathfrak{a}, m+2}(s)=-\sqrt{\frac{m^2+m+s(1-s)}{m^2+3m+2+s(1-s)}} \hat{E}_{\mathfrak{a}, m}(s).\end{equation}
It follows that it is enough to prove the theorem for $m=0$ and $m=1$.
\end{remark}

\subsection{Preliminaries}
By conjugating the group, we assume that $l$ lies on the imaginary axis, which we denote by $I$. We denote by $\hbox{len}(l)$ the hyperbolic length of $l$:
$$\hbox{len}(l):=\int_{l} 1 \; ds(z).$$
For $z=x+iy \in \mathbb{C}$, we define the \emph{Huber coordinates} $(u,v)$ as follows:
\begin{equation} \label{coordsys} u(z)=\log{|z|}, \; \; \; \; v(z)=-\arctan{\left(\frac{x}{y}\right)},\nonumber \end{equation}
or, equivalently,
$$x=-e^u\sin{v}, \; \; \; \; y=e^u\cos{v}. $$  
We note that, for $\gamma$ diagonal, we have $$u\left(\gamma z \right)=u(z)+\log{\lambda},$$ where $\lambda$ is the norm of $\gamma$, and
$$v\left(\gamma z\right)=v(z).$$
Note further that $v(z)$ can be interpreted as the anticlockwise angle formed between $z=ie^{u+iv}$ and the positive imaginary axis. 
The following two transforms are used in the study of the hyperbolic-hyperbolic case of the double coset counting problem:
\begin{align}d^{(0)}_{t}(f)&=\int_{0}^{\frac{\pi}{2}} \frac{1}{\cos^2{v}}f\left(\frac{1}{\cos^2{v}}\right)\cdot {}_2F_1\left(\frac{s}{2},\frac{1-s}{2} \, ; \frac{1}{2} \, ; -\tan^{2}{v} \right)\, dv, \label{d0def} \\
d^{(1)}_{t}(f)&=\int_{0}^{\frac{\pi}{2}}\frac{\tan^2{v}}{\cos^2{v}} f\left(\frac{1}{\cos^2{v}}\right) \cdot {}_2F_1\left(\frac{s+1}{2},\frac{2-s}{2} \, ; \frac{3}{2} \, ; -\tan^{2}{v} \right)\, dv,\label{d1def}\end{align}
where $s=1/2+it$. In this paper, instead of using the formulas (\ref{d0def}) and (\ref{d1def}) directly, we express $d^{(0)}_{t}$ and $d^{(1)}_{t}$  in terms of the Selberg/Harish-Chandra transform (see the proof of Proposition \ref{inversionprop}).

The transform $d_{t}^{(0)}(f)$ is called the Huber transform of $f$, and it appears in the spectral expansion of the \emph{Huber series} $A^{(0)}_f(z)$ (see \cite[eq.(4),(26)]{huber}), defined by
$$A^{(0)}_f(z):=\sum_{\gamma \in \Gamma_1 \backslash \Gamma}f\left(\frac{1}{\cos^2{\left( v(\gamma z)\right)}}\right).$$
This is an automorphization of $f$ composed with $1/\cos^2{\left( v(z)\right)}$, with respect to $\Gamma_1 \backslash \Gamma$, where $1/\cos^2{\left( v(z)\right)}$ is associated to the hyperbolic distance of $z$ from the imaginary axis, $\rho(z,I)$, via the relation $1/\cos^2{\left( v(z)\right)}=\cosh^2{\rho(z,I)}$.
Huber in \cite{huber} and Chatzakos--Petridis in \cite{chatzakos} study the elliptic-hyperbolic double coset counting problem using the spectral expansion of the Huber series. The first author in his thesis \cite{lekkas} used this spectral expansion to find a relative trace formula, which he used in the study of the hyperbolic-hyperbolic double coset problem. 

On the other hand, the transform $d_t^{(1)}(f)$ appears in the spectral expansion of the series $A_{f}^{(1)}(z)$, defined by
$$A^{(1)}_f(z):=\sum_{\gamma \in \Gamma_1 \backslash \Gamma} \tan{\left(v\left( \gamma z \right) \right)}f\left( \frac{1}{ \cos^2\left( v\left( \gamma z \right) \right)} \right).$$
In \cite{voskou}, the second author proves a refined version of the hyperbolic-hyperbolic double coset counting problem, using the spectral expansion of $A_f^{(1)}(z)$. This refined version is important for certain arithmetic applications regarding totally real quadratic fields. Such applications have also been studied by Hejhal in the series of papers \cite{hejhal1},\cite{hejhal2}, \cite{hejhal3}, without detailed proofs. 

In particular, we have the following spectral expansions.
\begin{theorem}[\cite{chatzakos}, \cite{voskou}] \label{spectralexpthm}
    Assume that $\Gamma$ is a Fuchsian Group of the first kind. For $f$ a continuous, piecewise differentiable function with exponential decay at infinity,
we have the following spectral expansions:
\begin{align*}
\textup{(a)}& \displaystyle & A^{(0)}_f(z)=&\;2\sum_{j}d^{(0)}_{t_j}(f)\hat{u}_{0,j}u_{0,j}(z)
-\sum_{\mathfrak{a}}\frac{i}{2\pi}\int_{\left(1/2\right)}d_t^{(0)}(f)\hat{E}_{\mathfrak{a},0}\left(s\right)\overline{E_{\mathfrak{a},0}\left(z,s\right)}ds
,  \\ 
\textup{(b)}&  \displaystyle & A^{(1)}_f(z)=&\;2\sum_{j}\sqrt{\lambda_j}d^{(1)}_{t_j}(f)\hat{u}_{1,j}u_{0,j}(z)
-\sum_{\mathfrak{a}}\frac{i}{2\pi}\int_{\left(1/2\right)}\sqrt{s(1-s)} \cdot d_t^{(1)}(f)\hat{E}_{\mathfrak{a},1}\left(s\right)\overline{E_{\mathfrak{a},0}\left(z,s\right)}ds
.
\end{align*}
\end{theorem}
\begin{remark}
    We note that, while $u_{1,j}(z)$ is not necessarily real-valued over $\mathbb{C}$, the periods $\hat{u}_{1,j}$ always are. Indeed, $$ u_{1,j}(z):=\frac{i}{\sqrt{\lambda_j}}K_0u_{0,j}(z)=\frac{ie^{-iv}}{\sqrt{\lambda_j}}\cos{v}  \left(\frac{\partial}{\partial u}-i\frac{\partial}{\partial v} \right)u_{0,j}(z)$$ and, hence, by periodicity of $u_{0,j}(z)$ with respect to the parameter $u$,
    $$\hat{u}_{1,j}:=\int_{l}u_{1,j}(z) ds(z)=\int_{0}^{ \textrm{len}(l)}u_{1,j}(z) \, du =\lambda_j^{-1/2} \int_{0}^{ \textrm{len}(l)}  \frac{\partial}{\partial v} u_{0,j}(z)\, du,$$ is real, as $u_{0,j}(z)$ is real valued. The lack of conjugating the second factor in the spectral expansions is, therefore, justified.
\end{remark}
A main ingredient in our proof is the following pair of relative trace formulae, used in the study of the hyperbolic-hyperbolic problem in \cite{lekkas} and \cite{voskou}.
\begin{theorem}[Relative Trace Formulae \cite{lekkas}, \cite{voskou}] \label{modtrace} Let $f$ be a real, continuous, piecewise differentiable function with exponential decay. Let $\varepsilon$ be equal to $1$ if $\Gamma$ has an element with zero diagonal entries, and $0$ otherwise. Then, we have
\begin{align*}
\hbox{\textup{(a)} }& \; \displaystyle & (1+\varepsilon) f(1)\mathrm{len}(l)+\sum_{\gamma \in \Gamma_1 \backslash \Gamma \slash \Gamma_1-\mathrm{id}} g_0\left(B(\gamma);f\right)&=&2&\sum_{j} d^{(0)}_{t_j}(f)\hat{u}^2_{0,j}&+E^{(a)}(f),\\
\hbox{\textup{(b)} }& \; \displaystyle &(1-\varepsilon) f(1)\mathrm{len}(l)+\sum_{\gamma \in \Gamma_1 \backslash \Gamma \slash \Gamma_1-\mathrm{id}} g_1\left(B(\gamma); f\right)&=&2&\sum_{j} \lambda_jd^{(1)}_{t_j}(f)\hat{u}^2_{1,j}&+E^{(b)}(f),
\nonumber
\end{align*}
where $B(\gamma)= ad+bc$,
 $\mathrm{len}(l)$ is the hyperbolic length of $l$, 
\begin{eqnarray}g_0(u;f)&:=&2\int_{ \sqrt{\mathrm{max}(u^2-1,0)}}^{\infty}  \frac{f\left( x^2+1  \right)}{\sqrt{x^2+1-u^2}}dx=\int_{ \mathrm{max}(u^2,1)}^{\infty}  \frac{f\left( t  \right)}{\sqrt{t-u^2}}\frac{dt}{\sqrt{t-1}},\nonumber \\
g_1(u;f)&:=&u \cdot g_0\left(u; \, f+2\sqrt{x-1}\cdot f'\right), \nonumber
\end{eqnarray}

and
\begin{eqnarray}&E^{(a)}(f):=-&\sum_{\mathfrak{a}} \frac{i}{2 \pi}  \int_{\left(1/2\right)} \di  \left|\hat{E}_{\mathfrak{a},0}\left(s \right) \right|^2  
\,ds, \nonumber \\
&E^{(b)}(f):=-&\sum_{\mathfrak{a}} \frac{i}{2 \pi}  \int_{\left(1/2\right)} s(1-s)\dii  \left|\hat{E}_{\mathfrak{a},1}\left(s \right) \right|^2  
\,ds. \nonumber \end{eqnarray}
\end{theorem}
\begin{proof}
For the proof of (a), for $\Gamma$ cocompact and $\varepsilon=0$, see \cite[\S 3.1]{lekkas}. The general case is similar.
For the proof of (b), see \cite[\S 3, \S 8]{voskou}.
\end{proof}
\begin{nnremark}
    It is worth noting that in the case $\varepsilon=1$ the second part of the theorem is trivial, as both sides of the equation are identically $0$. We can see this by considering the automorphism %
    $\gamma \leftrightarrow \gamma'\gamma$, where $\gamma'$ has zero diagonal entries.
\end{nnremark}
\begin{nnremark}
    For $|B(\gamma)|>1$, the quantity $\cosh^{-1}{B(\gamma)}$ is the hyperbolic distance of $l$ from $\gamma \cdot l$. The case $|B(\gamma)| \leq 1$ corresponds to the cases where $l$ and $\gamma \cdot l$ intersect. See for example \cite[Lemma 1]{martin}.
\end{nnremark}
\begin{nnremark}
 It is worth noting that the first relative trace formula does not take into account the sign of $B(\gamma)$, where the second one does. The sign of $B(\gamma)$ specifies the direction of $\gamma \cdot l$ and is importance, for instance, for certain arithmetic applications. For more details, see \cite{voskou}.
\end{nnremark}
\subsection{Summary}
In section \ref{inversionsection}, by writing the series $A^{(0)}_f(z)$ and $A^{(1)}_f(z)$ in terms of integrals of automorphic kernels (Lemma \ref{equivalence}) and, using known properties of the spectral coefficients of such kernels,
we deduce inversion formulas for $f$ in terms of $d^{(0)}_t(f)$ and $d^{(1)}_t(f)$ (Proposition \ref{inversionprop}). We use these inversion formulas to deduce Lemma \ref{keyestim}, a pair of useful upper bounds for $f$ with respect to $d^{(0)}_t(f)$ and $d^{(1)}_t(f)$ respectively. In section \ref{estsection}, we use Lemma \ref{keyestim} to establish some technical estimates that will be used later. In section \ref{pfofsieve}, we follow an argument similar to Chamizo \cite{chamizo1} to finish the proof of Theorem \ref{sievevar}. We use the relative trace formulae from Theorem \ref{modtrace} for particular choices of $d^{(0)}_t(f),d^{(1)}_t(f)$, and the bounds established in section \ref{estsection}. 
\section{Inversion Formulae}
\label{inversionsection}
In this section we explain the relation between the transforms $d_t^{(0)}(f), d_t^{(1)}(f)$ and the Selberg/Harish-Chandra transform. We use this relation to derive inversion formulas for $f$ in terms of $d_t^{(0)}(f)$ and $ d_t^{(1)}(f)$.

Consider the automorphic kernel
$$K(z,w)=\sum_{\gamma \in \Gamma} k(\gamma z,w),$$
where $k(z,w):=k(u(z,w))$ is a function of the fundamental point pair invariant, $$u(z,w)=\frac{|z-w|^2}{4\Im z\Im w}.$$
Define $$F(z \, ;\, \theta \, \vert \, k)=\int_{I\slash \Gamma_1}K(z,e^{i \theta}w) ds(w).$$
We will show that, for $k$ an appropriate transform of $f$, every Huber series $A^{(0)}_f(z)$ can be written in the form $F(z;\, 0 \, \vert \, k)$ and vice versa. Similarly, we show that every series $A^{(1)}_f(z)$ can be written in the form $F'(z;\, 0 \, \vert \, k)$, where the derivative is taken with respect to $\theta$.
\begin{lemma} \label{equivalence}
    For every $z \in \mathbb{H}$, we can relate the series $A_f^{(0)}(z)$ and $A_f^{(1)}(z)$ to integrals of automorphic kernels in the following way.
    \begin{enumerate}
        \item[\textup{(a)}] We have
    \begin{equation} \label{equiv1}  A^{(0)}_{f_0}(z)=F(z\,;\, 0 \, \vert \, k_0), \end{equation}
    where $f_0$ and $k_0$ are continuous, piecewise differentiable functions with exponential decay at infinity related via the formula
\begin{equation} \label{ktof1} f_0( p)=2\int_{\sqrt{p}}^{+\infty}k_0\left(\frac{x-1}{2}\right) \frac{dx}{\sqrt{x^2-p}}, \quad p\geq 1 ,\end{equation}
and, conversely,
\begin{equation} \label{ftok1} k_0(u)=-\frac{2u+1}{\pi} \int_{(2u+1)^2}^{+\infty}f_0'(p)\frac{dp}{\sqrt{p-(2u+1)^2}}. \end{equation}
\item[\textup{(b)}] We have \begin{equation} \label{equiv2} A^{(1)}_{f_1}(z)=F'(z\,;\, 0 \, \vert \, k_1), \end{equation}
    where $f_1$ and $k_1$ are continuous, piecewise differentiable functions with exponential decay at infinity, related via the formula
\begin{equation} \label{ktof2} f_1( p)=-\int_{\sqrt{p}}^{+\infty}k_1'\left(\frac{x-1}{2}\right) \frac{dx}{\sqrt{x^2-p}}, \quad p\geq 1 , \end{equation}
and, conversely,
\begin{equation} \label{ftok2} k_1'(u)=\frac{4u+2}{\pi} \int_{(2u+1)^2}^{+\infty}f_1'(p)\frac{dp}{\sqrt{p-(2u+1)^2}}. \end{equation}
\end{enumerate}
\end{lemma}
\begin{proof}
We have:
\begin{eqnarray}F(z \, ;\, \theta \, \vert \, k)&=&\sum_{\gamma \in \Gamma_1 \backslash \Gamma} \sum_{ \gamma_0 \in \Gamma_1}\int_{I \slash \Gamma_1}k(u(\gamma z, \gamma_0^{-1}e^{i\theta}w)) ds(w) \nonumber \\
&=&\sum_{\gamma \in \Gamma_1 \backslash \Gamma} \int_{I}k(u(\gamma z,e^{i\theta}w)) ds(w)=\sum_{\gamma \in \Gamma_1 \backslash \Gamma} \tilde{k}(\gamma z;\, \theta), \nonumber \end{eqnarray}
where
$$\tilde{k}( z; \, \theta):=\int_{I}k(u(z,e^{i\theta}w)) ds(w)=\int_{0}^{+\infty}k(u(z,ie^{i\theta}t)) \frac{dt}{t}.$$
Writing $z=x+iy$, we have
$$u(z,ie^{i\theta}t)=\frac{(x+t\sin{\theta})^2+(y-t\cos{\theta})^2}{4yt\cos{\theta}}=\left(\frac{x^2+y^2}{y^2} \cdot \frac{y}{4t}+\frac{t}{4y}\right)\cdot \sec{\theta}+\frac{x}{2y}\tan{\theta}-\frac{1}{2}.$$
In particular,
$$u(z,it)=\frac{x^2+y^2}{y^2} \cdot \frac{y}{4t}+\frac{t}{4y}-\frac{1}{2}, \quad \left.\frac{\partial}{\partial \theta}u(z,ie^{i\theta}t)\right|_{\theta=0}=\frac{x}{2y}=-\frac{1}{2}\tan{v(\gamma)}.$$
For (a), setting $r=t/y$, we have
$$\tilde{k}( z; \, 0)=\int_{0}^{+\infty}k\left(\frac{p}{4r}+\frac{r}{4}-\frac{1}{2} \right) \frac{dr}{r},$$ 
where $p=p(z)=(x^2+y^2)/y^2=1/\cos^2{v(z)}$. Hence, for
$f_0(p):=\tilde{k}_0( z; \, 0),$ we have

 $$A^{(0)}_{f_0}(z)=F(z\,;\, 0 \, \vert \, k_0).$$
 We are now left to verify the conversion formulae (\ref{ktof1}) and (\ref{ftok1}).

Under the change of variables $u=(p/2r+r/2)^2$ and $x=\sqrt{u}$ we get
$$f_0( p)=\int_{p}^{+\infty}k_0\left(\frac{\sqrt{u}-1}{2} \right) \frac{du}{\sqrt{u}\sqrt{u-p}}=2\int_{\sqrt{p}}^{+\infty}k_0\left(\frac{x-1}{2}\right) \frac{dx}{\sqrt{x^2-p}},$$
as required.
Using the inversion formula for Weyl integrals from \cite[Eq. 1.64,1.62]{iwaniec},
we have $$u^{-1/2} \cdot k_0\left(\frac{\sqrt{u}-1}{2} \right)=-\frac{1}{\pi}\int_{u}^{+\infty}f_0'(v)\frac{dv}{\sqrt{v-u}}.$$
In other words,
$$k_0(u)=-\frac{2u+1}{\pi} \int_{(2u+1)^2}^{+\infty}f_0'(v)\frac{dv}{\sqrt{v-(2u+1)^2}}. $$

For part (b), setting once again $r=t/y$, we have
$$\left.\frac{\partial}{\partial \theta}\tilde{k}( z; \theta)\right|_{\theta=0}=-\frac{1}{2} \tan{v(z)}\int_{0}^{+\infty}k'\left(\frac{p}{4r}+\frac{r}{4}-\frac{1}{2} \right) \frac{dr}{r}.$$
Hence, for $$f_1(p):=-\frac{1}{2}\int_{0}^{+\infty}k_1'\left(\frac{p}{4r}+\frac{r}{4}-\frac{1}{2} \right) \frac{dr}{r},$$ we have
 $$A^{(1)}_{f_1}(z)=F'(z\,;\, 0 \, \vert \, k_1).$$

 The conversion formulae (\ref{ktof2}) and (\ref{ftok2}) follow in a similar manner as in part (a).

\end{proof}
The spectral expansion of $F(z\,;\, 0 \, \vert \, k)$ and $F'(z\,;\, 0 \, \vert \, k)$ can be easily deduced from the well-known spectral expansion of $K(z,w)$ (see \cite[(7.17)]{iwaniec}). In particular,
let $h_k(t)$ be the Selberg/Harish-Chandra transform of $k$ (see \cite[Eq. 1.62]{iwaniec}). 
We have that
\begin{eqnarray} & F(z\,;\, 0 \, \vert \, k)=&\sum_{j}h_k(t_j)\hat{u}_{0,j}u_{0,j}(z)-\sum_{\mathfrak{a}}\frac{i}{4\pi}\int_{\left(1/2\right)}h_k(t)\hat{E}_{\mathfrak{a},0}\left(s\right)\overline{E_{\mathfrak{a},0}\left(z,s\right)}ds,\nonumber \\ \label{fkspectral2} & F'(z\,;\, 0 \, \vert \, k)=&\sum_{j}\sqrt{\lambda_j}h_k(t_j)\hat{u}_{1,j}u_{0,j}(z)-\sum_{\mathfrak{a}}\frac{i}{4\pi}\int_{\left(1/2\right)}  \sqrt{s(1-s)} \cdot h_k(t)\hat{E}_{\mathfrak{a},1}\left(s\right) \overline{E_{\mathfrak{a},0}\left(z,s\right)}ds. \nonumber
\end{eqnarray}
We combine these spectral formulas with Theorem \ref{spectralexpthm} and equate the spectral expansions of the two sides of the equations (\ref{equiv1}) and (\ref{equiv2}) from Lemma \ref{equivalence}. We then use the inversion formula for $h_k$ (see \cite[Eq. 1.64]{iwaniec}) to prove the following inversion formulae for $d_t^{(0)}(f)$ and $d_t^{(1)}(f)$.
\begin{proposition} \label{inversionprop}
    Let $f$ a continuous, piecewise differentiable function with exponential decay at infinity, $d_{t}^{(0)}(f)$ and $d_{t}^{(1)}(f)$ as in equations (\ref{d0def}) and (\ref{d1def}) respectively, and
    \begin{equation} \label{defofi}I(W,R):= \frac{-2\sqrt{2}}{\pi} \int_{0}^{1}\frac{\left( 2W+(R-W)y \right)^{-1/2}}{\sqrt{y(1-y)}} dy.
\end{equation}
We can recover $f$ from $d_{t}^{(0)}(f)$ and $d_{t}^{(1)}(f)$ in the following way:
    \begin{enumerate}
\item[\textup{(a)}] The function $f$ can be written in the form
    
    \begin{equation} \label{inversion1}f(\cosh^2{w})=\int_{w}^{+\infty} \omega'(\rho)I(\cosh{w},\cosh{\rho}) d\rho, \end{equation}
where \begin{equation}\label{omegadef}\omega(\rho)=\frac{1}{2\pi} \int_{-\infty}^{+\infty}e^{i \rho t}d^{(0)}_t(f)dt. \end{equation}
\item[\textup{(b)}] The function $f$ can be written in the form \begin{equation} \label{inv1} f( \cosh^2{w})=\int_{w}^{+\infty} \kappa'(\rho)I(\cosh{w},\cosh{\rho})d\rho ,\end{equation}
where
\begin{equation} \label{kappadef}\kappa(\rho)=\frac{i\left(\sinh{\rho}\right)^{-1}}{\pi} \int_{-\infty}^{+\infty}e^{i \rho t}td^{(1)}_t(f)dt. \end{equation}
\end{enumerate}
\end{proposition}
\begin{proof}
Using the inversion formula for $h_k$ (see \cite[Eq. 1.64]{iwaniec}), we have
\begin{equation} \label{kqinversion}k(u)=- \frac{1}{\pi}\int_{u}^{+\infty}\frac{q'(v)}{\sqrt{v-u}}dv, \end{equation}
where $$q(\sinh^2{(\rho/2)}):=\frac{1}{4\pi} \int_{-\infty}^{+\infty}e^{i\rho t}h(t)dt.$$
Using the substitution $v=\sinh^2{\left(\rho/2\right)}$, we rewrite this as
$$k\left(\frac{x-1}{2}\right)=- \frac{\sqrt{2}}{2\pi}\int_{\cosh^{-1}{x}}^{+\infty}\frac{\sinh{\rho} \cdot q'\left(\sinh^2{(\rho/2)}\right)}{\sqrt{\cosh{\rho}-x}}d\rho.$$

For part (a),
take $f_0=f$ and $k_0$ as in Eq.(\ref{ftok1}). 
Equating the spectral coefficients of both sides of  Eq.(\ref{equiv1}), we deduce that
$$h_{k_0}(t)=2d^{(0)}_{t}(f),$$
and, hence,
$$k_0\left(\frac{x-1}{2}\right)=- \frac{\sqrt{2}}{\pi}\int_{\cosh^{-1}{x}}^{+\infty}\frac{\omega'(\rho)}{\sqrt{\cosh{\rho}-x}}d\rho,$$
where $\omega(\rho)$ is given by Eq. (\ref{omegadef}).
We now substitute into Eq. (\ref{inversion1}), giving
$$f(p)=2\int_{\sqrt{p}}^{+\infty}k\left(\frac{x-1}{2}\right) \frac{dx}{\sqrt{x^2-p}}=\frac{-2\sqrt{2}}{\pi} \int_{\sqrt{p}}^{+\infty} \int_{\cosh^{-1}{x}}^{+\infty}\frac{\omega'(\rho)}{\sqrt{\cosh{\rho}-x}}d\rho \cdot \frac{dx}{\sqrt{x^2-p}}.$$
Changing the order of integration and letting $p=\cosh^2w$, we have
\begin{equation} f(\cosh^2{w})=\int_{w}^{+\infty} \omega'(\rho)I(\cosh{w},\cosh{\rho}) d\rho, \nonumber \end{equation}
where, for $W<R$, $$I(W,R)= \frac{-2\sqrt{2}}{\pi} \int_{W}^{R}\frac{1}{\sqrt{\left(R-x\right)\left(x^2-W^2\right)}} dx.$$
Using the substitution $x=W+(R-W)y$, we have $$I(W,R)=\frac{-2\sqrt{2}}{\pi} \int_{0}^{1}\frac{\left( 2W+(R-W)y \right)^{-1/2}}{\sqrt{y(1-y)}} dy,$$
as required. For the case $W=R$, we define $I(W,W)$ by continuity.

For part (b), we start by differentiating Eq. (\ref{kqinversion}) in the following manner. Using integration by parts on Eq. (\ref{kqinversion}), we have
$$k(u)= \frac{2}{\pi}\int_{u}^{+\infty}q''(v)\sqrt{v-u} \, dv.$$
Hence, by Leibniz integral rule,
$$k'(u)=- \frac{1}{\pi}\int_{u}^{+\infty}\frac{q''(v)}{\sqrt{v-u}}dv,$$
which after the substitution $v=\sinh^2{(\rho/2)}$ and $u=(x-1)/2$ becomes:
$$k'\left(\frac{x-1}{2}\right)=- \frac{\sqrt{2}}{2\pi}\int_{\cosh^{-1}{x}}^{+\infty}\frac{\sinh{\rho} \cdot q''\left(\sinh^2{(\rho/2)}\right)}{\sqrt{\cosh{\rho}-x}}d\rho.$$
Take $f_1=f$ and $k_1$ as defined in Eq. (\ref{ftok2}). Equating the spectral coefficients of both sides of Eq.(\ref{equiv2}), we deduce that
$$h_{k_1}(t)=2d^{(1)}_{t}(f),$$
and hence, similarly with part (a), we have
\begin{equation} f(\cosh^2{w})=\int_{w}^{+\infty} \kappa'(\rho)I(\cosh{w},\cosh{\rho}) d\rho, \nonumber \end{equation}
where $$\kappa(\rho)=q'(\sinh^2{(\rho/2)})=\frac{2}{\sinh{\rho}}\left(q(\sinh^2{(\rho/2)})\right)'=\frac{i\left(\sinh{\rho}\right)^{-1}}{2\pi} \int_{-\infty}^{+\infty}e^{i \rho t}td^{(1)}_t(f)dt,$$ as in Eq. (\ref{kappadef}).

\end{proof}
\begin{remark}
    Using the integral representation of the hypergeometric function (see \cite[9.111]{toisap}), we can show that $I(W,R)$ can be written in terms of a hypergeometric function. In particular,
    $$I(W,R)=-2W^{-1/2} \cdot {}_2F_1\left(\frac{1}{2},\frac{1}{2}; \, 1; \, \frac{W-R}{2W}\right).$$
\end{remark}

We conclude this section by using the inversion formulae from Proposition \ref{inversionprop}, to derive the following estimates.
\begin{lemma} \label{keyestim}
    Let $f$, $\omega$ and $\kappa$ as in Proposition \ref{inversionprop}. Furthermore, let $\tau(w):=\kappa(w) \sinh{w}$. Assume that, as $w$ tends to infinity, both  $\tau(w)$ and $\omega(w)$ tend to zero. Then, for every $w\geq0$, we have
    \begin{fleqn}[\parindent]
    \begin{alignat*}{7}
    \textup{(a)} \qquad  \displaystyle &&  
          f(\cosh^2{w})  &\ll \quad& &
|\omega(w)|+ \int_{w}^{+\infty} |\omega(\rho)| d\rho, &&
\\
\textup{(b)} \qquad  \displaystyle &&   \left(\sinh{w}\cdot f\left(\cosh^2{w}\right)\right)'&\ll \quad& &|\tau'(w)|+|\tau(w)|
+\int_{w}^{+\infty}|\tau(\rho)| d \rho. && \nonumber
    \end{alignat*}
    \end{fleqn}
\end{lemma}
\begin{proof}
For part (a), applying integration by parts on Eq. (\ref{inversion1}), we have
$$ f(\cosh^2{w})=-\omega(w)I(\cosh{w},\cosh{w})-\int_{w}^{+\infty} \omega(\rho)\sinh{\rho}\frac{\partial I}{\partial R}(\cosh{w},\cosh{\rho}) d\rho . $$
On the other hand, we have

\begin{equation}\label{ibound}I(W,R)= \frac{-2\sqrt{2}}{\pi} \int_{0}^{1}\frac{\left( 2W+(R-W)y \right)^{-1/2}}{\sqrt{y(1-y)}} dy \ll  \int_{0}^{1}\frac{W^{-1/2}}{\sqrt{y(1-y)}} dy \ll W^{-1/2}, \end{equation}
where we used the fact that $$\int_{0}^{1}\frac{1}{\sqrt{y(1-y)}} dy =\pi < +\infty.$$
By differentiating Eq.(\ref{defofi}) with respect to $R$ and proceeding in a similar manner, we also have
\begin{equation} \label{iderivbound}\frac{\partial I}{\partial R}(W,R)=\frac{\sqrt{2}}{\pi} \int_{0}^{1}y \cdot \frac{\left( 2W+(R-W)y \right)^{-3/2}}{\sqrt{y(1-y)}} dy \ll \int_{0}^{1}y \cdot \frac{W^{-3/2}}{\sqrt{y(1-y)}} dy \ll W^{-3/2}. \end{equation}
On the other hand,
\begin{eqnarray} \label{iderivbound2}\frac{\partial I}{\partial R}(W,R)&=&\frac{\sqrt{2}}{\pi} \int_{0}^{1}y \cdot \frac{\left( 2W+(R-W)y \right)^{-3/2}}{\sqrt{y(1-y)}} dy \nonumber \\ &\ll& \int_{0}^{1}\frac{y}{2W+(R-W)y} \cdot \frac{W^{-1/2}}{\sqrt{y(1-y)}} dy \ll (R-W)^{-1} W^{-1/2}. \end{eqnarray}
Combining eq. (\ref{iderivbound}) with eq. (\ref{iderivbound2}), we have
$$\frac{\partial I}{\partial R}(W,R) \ll R^{-1} W^{-1/2}. $$
 We deduce $$f(\cosh^2{w}) \ll
|\omega(w)|+ \int_{w}^{+\infty} |\omega(\rho)| d\rho
,$$ as required.

For part (b), applying integration by parts on Eq.(\ref{inv1}), and noting that
$$I(W,W)= \frac{-2\sqrt{2}}{\pi} \int_{0}^{1}\frac{\left( 2W \right)^{-1/2}}{\sqrt{y(1-y)}} dy =\frac{-2W^{-1/2}}{\pi} \int_{0}^{1}\frac{1}{\sqrt{y(1-y)}} dy =-2W^{-1/2},$$ 
we have
    $$\sinh{w}\cdot f\left(\cosh^2{w}\right)=-\frac{2}{\sqrt{\cosh{w}}}\cdot \tau(w)-\sinh{w}\int_{w}^{+\infty}\tau(\rho)\cdot \frac{\partial I}{\partial R}\left(\cosh{w}, \cosh{\rho}\right) d\rho.$$
    Therefore,
    \begin{eqnarray}\frac{d}{dw}\left(\sinh{w}\cdot f\left(\cosh^2{w}\right)\right)=&&\frac{\tanh{w}}{\sqrt{\cosh{w}}} \cdot \tau(w) -\frac{2}{\sqrt{\cosh{w}}}\cdot \tau'(w)\nonumber \\
    &-& \cosh{w}\int_{w}^{+\infty}\tau(\rho)\cdot \frac{\partial I}{\partial R}\left(\cosh{w},\cosh{\rho}\right) d\rho \nonumber \\
    &+&\sinh{w} \cdot \tau(w)\cdot \frac{\partial I}{\partial R}\left(\cosh{w},\cosh{w}\right) \nonumber \\
    &-& \sinh^2{w}\int_{w}^{+\infty}\tau(\rho)\cdot \frac{\partial^2 I}{\partial W\partial R}\left(\cosh{w},\cosh{\rho}\right) d\rho. \label{longexpression}
    \end{eqnarray}
    Differentiating Eq.(\ref{iderivbound}), we have 
    \begin{equation} \label{idderivbound}\frac{\partial^2 I}{\partial W \partial R}(W,R)=\frac{-3\sqrt{2}}{2\pi} \int_{0}^{1}y(2-y)\cdot \frac{\left( 2W+(R-W)y \right)^{-5/2}}{\sqrt{y(1-y)}} dy \ll W^{-5/2}. 
    \end{equation}
    Applying Eq. (\ref{iderivbound}) and Eq. (\ref{idderivbound}) into Eq. (\ref{longexpression}), we deduce that
    \begin{eqnarray}\frac{d}{dw}\left(\sinh{w}\cdot f\left(\cosh^2{w}\right)\right) \ll && \left(\cosh{w}\right)^{-1/2} \left( |\tau(w)|+|\tau'(w)|\right) \nonumber \\ &+&\cosh{w}\int_{w}^{+\infty}|\tau(\rho)|\left(\cosh{w}\right)^{-3/2}d\rho
    \nonumber \\
    &+&\cosh{w}\cdot |\tau(w)| \cdot (\cosh{w})^{-3/2} \nonumber \\
&+&\left(\cosh{w}\right)^2\int_{w}^{+\infty}|\tau(\rho)|\left(\cosh{w}\right)^{-5/2}d\rho, \nonumber
    \end{eqnarray}
    and, hence:
    $$\frac{d}{dw}\left(\sinh{w}\cdot f\left(\cosh^2{w}\right)\right) \ll  |\tau'(w)|+|\tau(w)|
+\int_{w}^{+\infty}|\tau(\rho)| d\rho.$$

\end{proof}
\section{Estimates for $f$} \label{estsection}
We now use Lemma \ref{keyestim} to establish some bounds for $f_0$ and $f_1$, for particular choices of $d^{(0)}_{t}(f_0)$ and $d^{(1)}_{t}(f_1)$ that will be important in section \ref{pfofsieve}. In particular,  the choices of the transforms are such that $d^{(0)}_{t}(f_0)$ and $(1/4+t^2)d^{(1)}_{t}(f_1)$ are smoothings of the indicator of $[-T,T]$ times the oscillation factor $\cos{rt}$. This is analogous to \cite[Lemma 3.1]{chamizo1}. 
\begin{lemma} \label{estlemma}.
    Let $T,r$ be positive with $T\ge 1$ and $r \ll 1$. Let $f_0, \; f_1$ be functions such that \begin{eqnarray}&d^{(0)}_{t}(f_0)=&e^{-t^2/4T^2}\cos(rt), \nonumber \\ 
    &d^{(1)}_{t}(f_1)=&\frac{1}{2t^2}e^{-t^2/4T^2}\left(1-e^{-t^2/4}\right)\cos\left(rt\right). \nonumber \end{eqnarray}
    The following inequalities hold:
\begin{enumerate}
    \item[\textup{(a)}]
    \begin{enumerate}
        \item[\textup{i)}] For $f_0$, we have
        $$f_0(\cosh^2{w}) \ll \left\{\begin{array}{ll} T 
    \cdot e^{-T^2(w-r)^2} +1, & \hbox{ for }\quad w\leq 2r, \\  T 
\cdot e^{-T^2(w-r)^2}, & \text{ for } \quad w \ge 2r. \end{array}\right.$$
In particular,
\begin{equation} \label{estimatef0of1} f_0(1) \ll T e^{-T^2r^2} +1. 
\end{equation}
        \item[\textup{ii)}] For $f_1$, we have
    $$ \left(\sinh{w}\cdot f_1(\cosh^2{w})\right)' \ll \left\{\begin{array}{ll} T 
    \cdot e^{-T^2(w-r)^2} +1, & \hbox{ for }\quad w\leq 2r, \\  T 
\cdot e^{-T^2(w-r)^2}+e^{-2(w-r)^2/3},  & \text{ for } \quad w \ge 2r . \end{array}\right.$$
In particular,
\begin{equation} \label{estimatef1of1} f_1(1) \ll T e^{-T^2r^2} +1. 
\end{equation}
    \end{enumerate}

    \item[\textup{(b)}] Moreover, we can show that, for $\sinh^{-1}u\ge2r$,
\begin{fleqn}[\leftmargin]
   \begin{alignat}{3}
       \textup{i)}&& \qquad & \displaystyle 
        \int_{u}^{\infty} \frac{f_0(x^2+1)}{\sqrt{x^2-u^2}}dx &{}\ll{}& {} Te^{-T^2r^2/2} \cdot u^{-2} \:, 
    \label{intf0large} \\
    \textup{ii)}&& \qquad & \displaystyle 
        \int_{u}^{\infty} \frac{\left(xf_1(x^2+1)\right)'}{\sqrt{x^2-u^2}}dx &{}\ll{}& {} \left(Te^{-T^2r^2/2}+e^{-r^2/2}\right) \cdot u^{-2} .
    \label{intf1large}
   \end{alignat}
   \end{fleqn}
    \item[\textup{(c)}] Furthermore, for $u \gg 1$, we have:
    \begin{fleqn}[\leftmargin]
    \begin{alignat}{4}
    &{}&\textup{i)}&& \qquad & \displaystyle
    \int_{0}^{\infty} \frac{f_0(x^2+1)}{\sqrt{x^2+u^2}}dx &{}\ll{}& {} 1, \label{geometricsmallf0} \\
   &{}&\textup{ii)}&& \qquad & \displaystyle
     \int_{0}^{\infty} \frac{\left(xf_1(x^2+1)\right)'}{\sqrt{x^2+u^2}}dx  &{}\ll{}& {} 1. \label{geometricsmallf1} 
    \end{alignat}
    \end{fleqn}
\end{enumerate}

\end{lemma}
\begin{proof}
We start with the estimates related to $f_0$:
\item[(a)- i)]
As
$$d_t(f_0)=e^{-t^2/4T^2}\cos(rt),$$
using Eq. \ref{omegadef} and standard properties of the Fourier Transform (in particular, \cite[3.1.8, 3.2.23]{bateman}), we have $$\omega(x)=\frac{T}{2 \sqrt{\pi}} \left(e^{-T^2(x-r)^2} +e^{-T^2(x+r)^2}\right).$$
For $x \geq 0$, we deduce $$ \omega(x) \ll Te^{-T^2(x-r)^2}.$$
Hence, by Lemma \ref{keyestim} (a),

\begin{eqnarray}f_0(\cosh^2{w}) &\ll& T
e^{-T^2(w-r)^2}+T \int_{w}^{+\infty} e^{-T^2(\rho-r)^2}d\rho
\nonumber \\ &=&
Te^{-T^2(w-r)^2}+ \int_{T(w-r)}^{+\infty} e^{-y^2}dy
.  \label{festimate1}
 \end{eqnarray}
For $w>2r$, using Eq. (\ref{festimate1}) and the fact that, for $x \geq 0$, the \emph{complementary error function} $\hbox{erfc}(x)$ satisfies \begin{equation}\hbox{erfc}(x):=\frac{2}{\sqrt{\pi}}\int_{x}^{+\infty} e^{-y^2} dy = \frac{2}{\sqrt{\pi}}\int_{0}^{+\infty} e^{-(x+u)^2} du \leq \frac{2}{\sqrt{\pi}}\cdot e^{-x^2}\cdot \int_{0}^{+\infty} e^{-u^2} du = e^{-x^2}, \label{erfc} \end{equation}
we have
\begin{eqnarray}f_0(\cosh^2{w}) &\ll& 
Te^{-T^2(w-r)^2}+e^{-T^2(w-r)^2}
\nonumber \\ &\ll& T
e^{-T^2(w-r)^2} . \label{festimate2}  \end{eqnarray}
On the other hand, for $w<2r$, Eq. (\ref{festimate1}) gives
\begin{equation} \label{festimate3}f_0(\cosh^2{w}) \ll T e^{-T^2(w-r)^2} +1. \end{equation}
\item[(b)- i)] For $ u> \sinh{2r}$, using Eq. (\ref{festimate2}), 
we have, for every $x \geq u$,
\begin{equation} \label{bismall} f_0(x^2+1) \ll Te^{-T^2(\sinh^{-1}x-r)^2} \ll Te^{-T^2r^2/2}\cdot e^{-T^2(\sinh^{-1}x-r)^2/2}. \end{equation}

On the other hand, for $x \geq 1$,
$$T^2(\sinh^{-1}x-r)^2/2 \geq T^2(\sinh^{-1}x)^2/8 \geq \left(\log{x}\right)^2/8,$$
giving
$$f_0(x^2+1) \ll Te^{-T^2r^2/2}\cdot e^{-\left(\log{x}\right)^2/8} \ll Te^{-T^2r^2/2} \cdot x^{-2}. $$
It is easy to see that this bound is still valid if $x<1$. Indeed, in that case, eq. (\ref{bismall}) gives
$$f_0(x^2+1) \ll Te^{-T^2r^2/2}\cdot e^{-T^2(\sinh^{-1}x-r)^2/2} \ll Te^{-T^2r^2/2} \ll Te^{-T^2r^2/2} \cdot x^{-2}. $$
Hence,
\begin{eqnarray}\int_{u}^{+\infty} \frac{f_0(x^2+1)}{\sqrt{x^2-u^2}} dx &\ll& Te^{-T^2r^2/2}\int_{u}^{+\infty} 
\frac{1}{x^2\sqrt{x^2-u^2}} dx \nonumber \\
&=& Te^{-T^2r^2/2}\cdot u^{-2} \cdot  \int_{1}^{+\infty} 
\frac{1}{y^2\sqrt{y^2-1}} dy
\nonumber \\
&\ll& Te^{-T^2r^2/2}\cdot u^{-2},
\label{lemmabb}
\end{eqnarray}
as required.
\item[(c)- i)]
Finally, for $u \gg 1$, 
\begin{equation}\int_{0}^{+\infty} \frac{f_0(x^2+1)}{\sqrt{x^2+u^2}} dx \ll  \int_{0}^{+\infty} \frac{|f_0(x^2+1)|}{\sqrt{x^2+1}} dx. \nonumber
\end{equation}
Using Eq. (\ref{festimate3}) for $0\leq x \leq \sinh{2r}$ and Eq. 
 (\ref{festimate2}) for $x \geq \sinh{2r}$,
\begin{eqnarray}\int_{0}^{+\infty} \frac{f_0(x^2+1)}{\sqrt{x^2+u^2}} dx &\ll& \int_{0}^{\sinh{2r}} 1 dx +T\int_{0}^{+\infty} \frac{1}{\sqrt{x^2+1}}\cdot e^{-T^2(\sinh^{-1}{x}-r)^2} dx \nonumber \\
&\ll& 1+T\int_{0}^{+\infty}  e^{-T^2(w-r)^2} dw \ll 1, \end{eqnarray}
as required.

For the case of $f_1$, we note that
$$\tau(\rho)=\frac{i}{\pi} \int_{-\infty}^{+\infty}e^{i \rho t}td^{(1)}_t(f_1)dt=\frac{i}{2\pi} \int_{-\infty}^{+\infty}e^{i \rho t}e^{-t^2/4T^2}\left(\frac{1-e^{-t^2/4}}{t}\right)\cos\left(rt\right)\, dt.$$
Using standard properties of the Fourier transform (in particular, \cite[3.1.8, 3.1.10, 3.2.23]{bateman}), we deduce that
\begin{equation}\label{tau}\tau(\rho)=\frac{\sqrt{\pi}}{4}\left(\hbox{erfc}\left(T\left(\rho-r\right) \right)-\hbox{erfc}\left( B \left(\rho-r\right)\right)+\hbox{erfc}\left(T\left(\rho+r\right) \right)-\hbox{erfc}\left( B\left(\rho+r\right)\right) \right), \end{equation}
where
$B=T/\sqrt{T^2+1}=1+o(1)$.
Using the inequality (\ref{erfc}) for $x>0$, we have, for $\rho>r$,
\begin{equation} \label{tauineq}\tau(\rho) \ll e^{-T^2(\rho-r)^2}+e^{-2(\rho-r)^2/3}.\end{equation}
On the other hand, as $\hbox{erfc}(x)$ is bounded, we have
\begin{equation} \label{taubdd}\tau(\rho)\ll 1,\end{equation} for every $\rho \geq 0$.
Furthermore, differentiating Eq. (\ref{tau}) gives
$$\tau'(\rho)=-\frac{1}{2}\left(Te^{-T^2\left(\rho-r\right)^2}-Be^{-B^2\left(\rho-r\right) ^2}+Te^{-T^2\left(\rho+r\right)^2}-Be^{-B^2\left(\rho+r\right)^2}\right),$$
and hence
\begin{equation} \label{tauderineq}\tau'(\rho) \ll Te^{-T^2\left(\rho-r\right)^2}+e^{-2\left(\rho-r\right) ^2/3}. \end{equation}
\item[(a)- ii)] Using Eq.(\ref{tauineq}), and Eq. (\ref{tauderineq}),  Lemma \ref{keyestim} b) gives that, for $w>r$,
\begin{eqnarray}\left(\sinh{w}\cdot f_1\left(\cosh^2{w}\right)\right)'&\ll& Te^{-T^2\left(w-r\right)^2}+e^{-2\left(w-r\right) ^2/3}
+\int_{w}^{+\infty}\left(Te^{-T^2\left(\rho-r\right)^2}+e^{-2\left(\rho-r\right) ^2/3}\right) d \rho \nonumber \\
&\ll& Te^{-T^2\left(w-r\right)^2}+e^{-2\left(w-r\right) ^2/3}, \nonumber
\end{eqnarray}
as required.

On the other hand, for $w<r$, by Eq. (\ref{tauineq}), Eq. (\ref{taubdd}) and Eq. (\ref{tauderineq}), Lemma \ref{keyestim} b) gives
\begin{align} &\left(\sinh{w}\cdot f_1\left(\cosh^2{w}\right)\right)'\nonumber \\
& \qquad \ll {} Te^{-T^2\left(w-r\right)^2}+e^{-2\left(w-r\right) ^2/3}+1
+\int_{w}^{r}1 d \rho \nonumber+\int_{r}^{+\infty}\left(Te^{-T^2\left(\rho-r\right)^2}+e^{-2\left(\rho-r\right) ^2/3}\right) d \rho \nonumber \\
&\qquad \ll {} Te^{-T^2\left(w-r\right)^2}+1+r+\hbox{erfc}(0) \ll Te^{-T^2\left(w-r\right)^2}+1, \nonumber
\end{align}
as required.

(b)-ii) and (c)-ii) follow from (a)-ii) in the same manner as (b)-i),(c)-i) followed from (a)-i).
\end{proof}
\begin{remark}
Alternatively, the result can also be proved by applying \cite[Lemma 5.1]{chamizok} to Lemma \ref{equivalence}.
\end{remark}
\section{Proof of the Sieve Inequality} \label{pfofsieve}
Applying the relative trace formula \cite[(3.2)]{lekkas} for the function $f$ of Lemma \ref{estlemma}, and using Lemma \ref{estlemma}, we now prove Theorem \ref{sievevar}.
\begin{proof}[Proof of Theorem \ref{sievevar}]
We first derive eq. (\ref{highperiods}), from which it follows that it is enough to prove the cases $m=0$ and $m=1$. Let $c_{m}:=\sqrt{\lambda_j+m^2+m}.$
By definition, we have
\begin{equation}c_{m+1}u_{m+2,j}=i K_{m+1}u_{m+1,j}. \label{eq1step1} \end{equation}
On the other hand, by \cite[eq.(3)\&(8)]{fay}, we have
\begin{equation}c_mu_{m,j}=i\overline{K_{-m-1}}u_{m+1,j}. \label{eq1step2}
\end{equation}
Adding eq. (\ref{eq1step1}) to eq. (\ref{eq1step2}), and using the fact that
$$K_{m+1}= iy\left( \frac{\partial}{\partial x}-i\frac{\partial}{\partial y}\right)+(m+1),$$ we deduce that
$$ c_{m+1}u_{m+2,j}+c_mu_{m,j}=2iy\frac{\partial}{\partial y}u_{m+1,j}.$$
Hence, integrating along $l$ with respect to $ds(z)$, we have
\begin{eqnarray}c_{m+1}\hat{u}_{m+2,j}+c_m \hat{u}_{m,j}&=&\int_{l}2iy\frac{\partial}{\partial y}u_{m+1,j}(z) ds(z) \nonumber \\
&=& 2i \int_{1}^{\lambda}\frac{\partial}{\partial y}u_{m+1,j}(iy) dy \nonumber \\
&=& 2i\left( u_{m+1,j}(\lambda \cdot i)-u_{m+1,j}(i) \right)=0, \nonumber 
\end{eqnarray}
where $\lambda=\exp\left(\hbox{len}(l)\right)$. The last equality follows from periodicity of $u_{m+1,j}$ in the $u$ variable. Hence, we have
$$c_{m+1}\hat{u}_{m+2,j}=-\frac{c_m}{c_{m+1}} \hat{u}_{m,j}=-\sqrt{\frac{m^2+m+\lambda_j}{m^2+3m+2+\lambda_j}}\hat{u}_{m,j},$$ as required.
The relation for the periods of Eisenstein series follows in the same way.

For $m=0$, let $$M= \sum_{\nu=1}^R \Big| \sum_{|t_j|\le T}a_jx_\nu^{it_j}\widehat{u}_{0,j}+\frac{1}{4\pi}\sum_{\mathfrak{a}}\int_{-T}^{T}a_{\mathfrak{a}}(t)x_{\nu}^{it}\hat{E}_{\mathfrak{a}, 0}(1/2+it)\; dt\Big|^2 \; .$$
By duality there exists a unit complex vector $\mathbf{b}=(b_1,b_2, \dots ,b_R),$ such that
$$M=\Bigg(\sum_{\nu=1}^R b_{\nu}\left|\sum_{|t_j|\le T}a_jx_\nu^{it_j}\widehat{u}_{0,j}+\frac{1}{4\pi}\sum_{\mathfrak{a}}\int_{-T}^{T}a_{\mathfrak{a}}(t)x_{\nu}^{it}\hat{E}_{\mathfrak{a}, 0}(1/2+it)\; dt\right|\Bigg)^2 \; .$$
After changing the order of summation and applying the Cauchy--Schwarz inequality on the space $\mathbb{C}^R$, we get
$$ M\ll ||\mathbf{a}||_*^2\Tilde{M} \; ,$$
where $\Tilde{M}$ is defined as
$$ \Tilde{M} = \sum_{|t_j|\le T}\Big|\sum_{\nu=1}^R b_{\nu}x_{\nu}^{it_j}\widehat{u}_{0,j}\Big|^2+\frac{1}{4\pi}\sum_{\mathfrak{a}}\int_{-T}^{T}\Big|\sum_{\nu=1}^R b_{\nu}x_{\nu}^{it}\hat{E}_{\mathfrak{a}, 0}(1/2+it)\Big|^2 \;dt.$$
We extend the summation for $|t_j|<T$ and integration for $-T\leq t \leq T$ to the whole spectrum using smooth weights that approximate the indicator of $|t|<T$ from above. In particular, we use weights of the form $\hbox{exp}(-t^2/4T^2)$, which decay rapidly at infinity and are bounded away from zero in the interval $[-T,T]$. We have 
$$ \Tilde{M} \ll \sum_j \ \exp\left(-t_j^2/4T^2\right)\Big|\sum_{\nu=1}^R b_{\nu}x_{\nu}^{it_j}\widehat{u}_{0,j}\Big|^2 +\sum_{\mathfrak{a}}\int_{-\infty}^{+\infty}\exp\left(-t^2/4T^2\right)\Big|\sum_{\nu=1}^R b_{\nu}x_{\nu}^{it}\hat{E}_{\mathfrak{a}, 0}(1/2+it)\Big|^2 \;dt.$$
After opening the squares and changing the order of summation we get
\begin{equation} \label{mineq1} M \ll ||\mathbf{a}||_*^2\max_{\mu \in \{1,2, \dots , R\}}\sum_{\nu=1}^R\big|S_{\nu\mu}\big| \; ,\end{equation} where we define
$$S_{\nu\mu}= \sum_j \exp\left(-t_j^2/4T^2\right)\cos(r_{\nu\mu}t_j)\widehat{u}_{0,j}^2+\frac{1}{4 \pi}\sum_{\mathfrak{a}}\int_{-\infty}^{+\infty}\exp\left(-t^2/4T^2\right)\cos{(r_{\mu \nu}t)}\Big|\hat{E}_{\mathfrak{a}, 0}(1/2+it)\Big|^2 \;dt,$$
and $$ r_{\nu\mu}=|\log(x_\nu/x_\mu)| \; .$$
Since $X\le x_{\nu},x_{\mu} \le 2X$, we get
$$r_{\nu\mu}=|\log(x_{\nu}/x_{\mu})|\le \log 2 \: $$
and, for $\mu \neq \nu$, by the mean value theorem for $\log{x}$, \begin{equation} \label{rjbound} r_{\nu\mu} \geq \frac{|x_{\nu}-x_{\mu}|}{\hbox{max}\left(x_{\nu},x_{\mu}\right)} \geq \frac{(j+1)\delta}{2X} , \end{equation}
where $j$ is the number of $x_i$'s between $x_{\nu}$ and $x_{\mu}$.
Fixing $r=r_{\mu \nu}$, let $S:=S_{\mu \nu}$.
From Theorem \ref{modtrace} a), we have
$$S \ll f_0(1) \hbox{len}(l)+\sum_{\gamma \in \Gamma_1 \backslash \Gamma \slash \Gamma_1-\mathrm{id}} 2\int_{ \sqrt{\mathrm{max}(B^2(\gamma)-1,0)}}^{\infty}  \frac{f_0\left( x^2+1  \right)}{\sqrt{x^2+1-B^2(\gamma)}}dx,$$
where $f_0$ is as in Lemma \ref{estlemma}.

We can partition 
$x_1,\dots,x_R$ into finitely many subsets such that each subset lies in an interval of the form  $[\rho^i X, \rho^{i+1}X]$ , where $\rho=1+\epsilon$, with $\epsilon>0$ a small fixed number. Hence, we can assume that $r_{\mu \nu}$ is smaller than any fixed constant (namely, smaller than $\log{\rho}$). In particular, we can assume that there is no $\gamma$ with $1 < |B(\gamma)|\leq \cosh{2r}$.
Therefore, using Lemma \ref{estlemma}, and in particular Eq. (\ref{estimatef0of1}) and Eq. (\ref{geometricsmallf0}), we deduce that
$$S \ll 1+ Te^{-r^2T^2} +\sum_{|B(\gamma)|>\cosh{2r}} \int_{\sqrt{B^2(\gamma)-1}}^{+\infty} \frac{f_0(x^2+1)}{\sqrt{x^2+1-B^2(\gamma)}} dx.$$
From Lemma \ref{estlemma}(b), we have
$$\sum_{|B(\gamma)|>\cosh{2r}} \int_{\sqrt{B^2(\gamma)-1}}^{+\infty} \frac{f_0(x^2+1)}{\sqrt{x^2+1-B^2(\gamma)}} dx \ll Te^{-T^2r^2/2} \cdot 
 \sum_{|B(\gamma)|>\cosh{2r}}(B^2(\gamma)-1)^{-1}. $$
 Via partial summation, and the fact that $ \#\left\{ \gamma \in \Gamma_1 \backslash \Gamma \slash \Gamma_1 \vert \left| B(\gamma) \right| \leq X \right\} \ll X $ (see, for example, \cite[Theorem 1.2]{lekkas}),
 we deduce that
 \begin{equation}\label{geometricbig}\sum_{|B(\gamma)|>\cosh{2r}} \int_{\sqrt{B^2(\gamma)-1}}^{+\infty} \frac{f_0(x^2+1)}{\sqrt{x^2+1-B^2(\gamma)}} dx \ll Te^{-T^2r^2/2}, \nonumber  \end{equation}
and, therefore,
$$S \ll 1+Te^{-T^2r^2/2}.$$
Hence, for fixed $\mu$,
\begin{equation} \label{sumofsbound}\sum_{\nu=1}^{R}|S_{\mu \nu}| \ll R+T\sum_{\nu=1}^{R}e^{-T^2r_{\mu \nu}^2/2} \ll \delta^{-1}X+T\sum_{\nu=1}^{R}e^{-T^2r_{\mu \nu}^2/2}. \end{equation}
Furthermore, from Eq. (\ref{rjbound}), we have
\begin{eqnarray}\sum_{\nu=1}^{R}e^{-T^2r_{\mu \nu}^2/2} &\ll& 1+ \sum_{j=1}^{R-1} \exp\left(-\frac{j^2}{8}\cdot \left(\frac{T\delta}{X}\right)^2 \right) \nonumber \\
&\ll&
1+ \sum_{j < 2\sqrt{2}X/\delta T} 1+\sum_{j \geq 2\sqrt{2}X/\delta T }\frac{8}{j^2}\cdot \left(\frac{X}{T \delta }\right)^2\nonumber \\
&\ll& 1+  \frac{X}{\delta T }+\left(\frac{X}{T \delta }\right)^2\sum_{j \geq 2\sqrt{2}X/\delta T }\frac{1}{j^2}\nonumber \\
&\ll& 1+  \frac{X}{\delta T }+\left(\frac{X}{T \delta }\right)^2 \cdot \frac{\delta T}{X} \ll  1+  \frac{X}{\delta T }. \label{sumfogauss}
\end{eqnarray}
Combining Eq (\ref{sumofsbound}) and (\ref{sumfogauss}), we have
\begin{equation} \sum_{\nu=1}^{R}|S_{\mu \nu}| \ll \delta^{-1}X+T\cdot \left(1+ \frac{X}{\delta T} \right) \ll T+\delta^{-1}X.  \nonumber \end{equation}
Via Eq. (\ref{mineq1}), this concludes the proof of the theorem.

For $m=1$, for simplicity, assume $\Gamma$ cocompact (for the general cofinite case, compare with the proof of $m=0$ ).
In a similar manner as in the case $m=0$,
let $$M= \sum_{\nu=1}^R \Big| \sum_{|t_j|\le T}a_jx_\nu^{it_j}\widehat{u}_{1,j}\Big|^2 \; .$$
Once again, by duality there exists a unit complex vector $\mathbf{b}=(b_1,b_2, \dots ,b_R),$ such that
$$M=\Bigg(\sum_{\nu=1}^R b_{\nu}\sum_{|t_j|\le T}a_jx_\nu^{it_j}\widehat{u}_{1,j}\Bigg)^2 \; .$$
After changing the order of summation and applying the Cauchy--Schwarz inequality on the space $\mathbb{C}^R$, we get
$$ M\ll ||\mathbf{a}||_*^2\Tilde{M} \; ,$$
where $\Tilde{M}$ is defined as
$$ \Tilde{M} = \sum_{|t_j|\le T}\Big|\sum_{\nu=1}^R b_{\nu}x_{\nu}^{it_j}\widehat{u}_{1,j}\Big|^2 \;.$$
In a similar manner as in the case $m=0$, we extend the sum interval to the whole spectrum. In order to apply Theorem \ref{modtrace} b), we need we want weights $\lambda_j=|t_j|^2+1/4$ to appear. For that end, we modify the smooth coefficients chosen as follows:
$$ \Tilde{M} \ll \sum_j \ \frac{\lambda_j}{t_j^2}\exp\left(-t_j^2/4T^2\right)\left(1-\exp\left(-t_j^2/4\right)\right)\Big|\sum_{\nu=1}^R b_{\nu}x_{\nu}^{it_j}\widehat{u}_{1,j}\Big|^2 \; .$$
After opening the squares and changing the order of summation we get
\begin{equation} \label{mineq2} M \ll ||\mathbf{a}||_*^2\max_{\mu \in \{1,2, \dots , R\}}\sum_{\nu=1}^R\big|S_{\mu\nu}\big| \; , \nonumber\end{equation} where we define
$$S_{\mu\nu}= \sum_j \frac{\lambda_j}{t_j^2}\exp\left(-t_j^2/4T^2\right)\left(1-\exp\left(-t_j^2/4\right)\right)\cos(r_{\nu\mu}t_j)\widehat{u}_{1,j}^2. \;  $$
From Theorem \ref{modtrace} b), we have
$$S \ll f_1(1)\hbox{len}(l)+\sum_{\gamma \in \Gamma_1 \backslash \Gamma \slash \Gamma_1-\mathrm{id}} 2B(\gamma)\int_{ \sqrt{\mathrm{max}(B^2(\gamma)-1,0)}}^{\infty}  \frac{\left(xf_1\left( x^2+1  \right)\right)'}{\sqrt{x^2+1-B^2(\gamma)}}dx,$$
where $f_1$ is as in Lemma \ref{estlemma}.
The rest of the proof follows similarly with $m=0$, using Eq. (\ref{estimatef1of1}), (\ref{geometricsmallf1}) and (\ref{intf1large}) in place of Eq. (\ref{estimatef0of1}), (\ref{geometricsmallf0}) and (\ref{intf0large}).
\end{proof}

\section{Acknowledgements} Both authors would like to thank their PhD advisor, Yiannis Petridis, for suggesting the problem and for his advice and patience. They would also like to thank the anonymous referee for their helpful comments and suggestions. The second author was supported by University College London and the
Engineering and Physical Sciences Research Council (EPSRC) studentship grant EP/V520263/1. The second author was also supported by the Swedish Research Council under grant no. 2016-06596
while in residence at Institut Mittag-Leffler in Djursholm, Sweden during the Analytic Number Theory
programme in the spring of 2024.

\section*{Data Availability}
Data sharing is not applicable to this article as no datasets were generated or
analysed for this work. The authors have no conflicts of interest, financial or non-financial, to
declare which are relevant to the content of this article.


\begin{thebibliography}{99}
\bibitem{balog} A. Balog, A. Biró, G. Harcos, P. Maga.
\emph{The prime geodesic theorem in square mean},
Journal of Number Theory,
Volume 198,
p.239-249, 2019.
\bibitem{bateman} H. Bateman, Bateman Manuscript Project. \emph{Tables of Integral Transforms.} New York: McGraw-Hill Book Company, 1954.
\bibitem{bombieri} E. Bombieri, \emph{Le grand crible dans la th\'eorie analytique des nombres,} Soc.
Math. France, Ast\'erisque 18, 1974.
\bibitem{chamizo1} F. Chamizo. \emph{The large sieve in riemann surfaces.} Acta Arithmetica, 77, no.4, 303–
313, 1996.
\bibitem{chamizo2} F. Chamizo. \emph{Some applications of large sieve in Riemann surfaces.} Acta Arithmetica, 77, no.4, 315–337, 1996.
\bibitem{chamizok}
F. Chamizo, D. Raboso.
\emph{On the Kuznetsov formula}, J. Funct. Anal., 268(4):869-886, 2015.
\bibitem{chatzakos}
D. Chatzakos, Y. Petridis. 
\emph{The hyperbolic lattice point problem in conjugacy classes.} Forum Math., 28:981–1003, 2016.
\bibitem{cherubini} G. Cherubini, J. Guerreiro. \emph{Mean square in the prime geodesic theorem}, Algebra Number Theory, 12, no.3, 571–597, 2018.
\bibitem{fay} J. D. Fay.
\emph{Fourier coefficients of the resolvent for a Fuchsian group.} J. Reine Angew. Math. 293/294, 143--203, 1977.
\bibitem{good} A. Good.
\emph{Local analysis of Selberg's trace formula. Lecture Notes in Mathematics, 1040.} Springer-Verlag, Berlin, 1983. i+128 pp.
\bibitem{toisap} I. S. Gradshteyn, I. M. Ryzhik. 
\emph{Table of integrals, series, and products. Translated from the Russian.} Translation edited and with a preface by Alan Jeffrey and Daniel Zwillinger. Seventh edition. Elsevier/Academic Press, Amsterdam, 2007. xlviii+1171 pp.
\bibitem{hejhal1} D. A. Hejhal. \emph{Sur certaines s\'{e}ries de {D}irichlet associ\'{e}es aux
              g\'{e}od\'{e}siques ferm\'{e}es d'une surface de {R}iemann
              compacte}
 CR Acad. Sci. Paris S{\'e}r. I. 294, p.273--276, 1982.
\bibitem{hejhal2} D. A. Hejhal. \emph{Sur quelques propri\'{e}t\'{e}s asymptotiques des p\'{e}riodes
              hyperboliques et des invariants alg\'{e}briques d'un
              sous-groupe discret de {${\rm PSL}(2,\,{\bf R})$}}
 CR Acad. Sci. Paris S{\'e}r. I. 294, p.509--512, 1982.
\bibitem{hejhal3} D. A. Hejhal. \emph{Quelques exemples de s\'{e}ries de {D}irichlet dont les
              p\^{o}les ont un rapport \'{e}troit avec les valeurs propres
              de l'op\'{e}rateur de {L}aplace-{B}eltrami hyperbolique}
 CR Acad. Sci. Paris S{\'e}r. I. 294, p.637--640, 1982.
 \bibitem{huber} H. Huber.
\emph{Ein Gitterpunktproblem in der hyperbolischen Ebene.} J. Reine Angew. Math. 496, 15--53, 1998. 
\bibitem{iwaniec}
H. Iwaniec. 
\emph{Spectral Methods of Automorphic Forms Second Edition.} Graduate
Studies in Mathematics, 53. , Revista Matematica Iberoamericana,
Madrid. American Mathematical Society, 2002.
\bibitem{jutila} M. Jutila. \emph{On spectral large sieve inequalities.} Funct. Approx. Comment. Math. 28, 7–18, 2000.
\bibitem{lekkas} D. Lekkas.
\emph{A relative trace formula and counting
geodesic segments in the hyperbolic
plane.}
Doctoral thesis (Ph.D), UCL (University College London), 2023.

\bibitem{martin} K. Martin, M. Mckee, E. Wambach. \emph{A relative trace formula for a compact
Riemann surface}. International Journal of Number Theory, 07(02), p.389–429,
2011.
\bibitem{nordentoft} A.C. Nordentoft, Y. Petridis, M.S. Risager. \emph{Bounds on shifted convolution sums for Hecke eigenforms.} Res. number theory 8, 26, 2022. 
\bibitem{roelcke} W. Roelcke. \emph{Das Eigenwertproblem der automorphen formen in der hyperbolischen Ebene, I.} Mathematische Annalen 167, 292-337, 1966. 
\bibitem{voskou} M. Voskou.   \emph{Refined Counting of Geodesic Segments in the Hyperbolic Plane.} preprint (2024), \url{https://doi.org/10.48550/arXiv.2407.03134}.









 
\end{thebibliography}
\end{document}